\documentclass{article}

\usepackage{amsmath, amssymb, enumerate}
\usepackage{graphicx}
\usepackage{amsmath}
\usepackage{amsthm}
\usepackage{amstext}
\usepackage{amssymb}
\usepackage[all]{xy}
\newcommand \R{\mathbb{R}}

\newcommand \C{\mathcal{C}}
\newcommand \Z{\mathbb{Z}}
\newcommand \N{\mathbb{N}}

\newcommand \D{\mathcal{D}}
\newcommand \Se{\mathcal{S}}

\newcommand \A{\mathfrak{A}}

\newtheorem{thm}{Theorem}
\newtheorem*{thrm}{Theorem}
\newtheorem{coro}[thm]{Corollary}

\newtheorem{lem}[thm]{Lemma}

\newtheorem{prop}[thm]{Proposition}
\newtheorem*{propo}{Proposition}
\theoremstyle{definition}
\newtheorem{defs}[thm]{Definition}
\newtheorem*{quest}{Question}
\newtheorem{ex}[thm]{Example}

\title {Complete Reducibility in Euclidean Twin Buildings}
\author {Denise K. Dawson}

\begin{document}

\maketitle

\begin{abstract}
In \cite{Se04}, J.P. Serre defined completely reducible subcomplexes of spherical buildings in order to study subgroups of reductive algebraic groups.  This paper begins the exploration of how one may use a similar notion of completely reducible subcomplexes of twin buildings to study subgroups of algebraic groups over a ring of Laurent polynomials and Kac-Moody groups. In this paper we explore the definitions of convexity and complete reducibility in twin buildings and some implications of the two in the Euclidean case.
\end{abstract}

\section{Introduction} Buildings were introduced by J. Tits as a geometric tool for studying certain algebraic groups over a field. A building can be thought of as a simplicial complex which is obtained by gluing together subcomplexes called apartments, which are made up of chambers (the simplices of maximal dimension) satisfying certain axioms. The apartments of a building are all isomorphic to a Coxeter complex.  For example, consider the reflection group $D_{2m}=\left<s,t|s^2=t^2=(st)^m=1\right>$. The elements of $D_{2m}$ act on the plane and we can consider the set of hyperplanes corresponding to the reflections.  By cutting the unit circle by these hyperplanes we get a decomposition of the circle into simplices, and this simplicial complex is a spherical Coxeter complex.  If $m=3$ then the simplicial complex will be a hexagon.

We can construct a building associated to $GL_n(k)$ for a field $k$ as follows.  Let $k$ be a field and let $\Delta(k^n)$ be the abstract simplicial complex with vertices being the nonzero proper subspaces of $k^n$, and with the maximal simplices being the chains $V_1<V_2<\ldots,<V_{n-1}$ of such subspaces. Then $\Delta(k^n)$ is a building and any basis of $k^n$ yields an apartment. This apartment consists of the vertices which correspond to subspaces spanned by proper nonempty subsets of the basis, and the simplices correspond to chains of these subspaces. For example, if $n=3$ and $\{e_1,e_2,e_3\}$ is any basis for $k^3$, then we get an apartment of $\Delta(k^3)$. The vertices correpond to the six proper nonempty subsets and the one-dimensional simplices correspond to chains of these subsets, hence we have the hexagon mentioned above.  Since the Coxeter complex is spherical, this is called a spherical building.

In spherical Coxeter complexes there is a bounded distance between any two points so there is a natural idea of opposite vertices and hence opposite chambers, which leads to many interesting properties of spherical buildings.  In buildings of nonspherical type (e.g. Euclidean buildings), there is no bound on the distance between any two vertices so there is no notion of opposition. 

Twin buildings were introduced by M. Ronan and J. Tits as a tool for studying groups of Kac-Moody type. They arise from these groups much like spherical buildings arise from algebraic groups and they extend to nonspherical buildings some of the ideas of spherical buildings, such as opposition. A twin building consists of a pair of buildings $(\C_+,\C_-)$ of the same type with an \emph{opposition relation} between the chambers of the two components. 

One consequence of the existence of opposites in spherical buildings is that one can use properties of the building to study completely reducible subgroups of a group $G$ which acts on a spherical building. In \cite{Se04}, J.P. Serre gives a definition for a completely reducible subgroup of a reductive algebraic group which generalizes the definition of a completely reducible representation and uses the existence of opposite simplices in the corresponding spherical buildings. His definition in terms of opposite simplices can be extended to a definition of complete reducibility in twin buildings.

Recall that if $V$ is a representation of a group $G$ then $V$ is \emph{completely reducible} if and only if for every proper $G$-invariant subspace $W$ of $V$ there is a proper $G$-invariant subspace $W'$ such that $W\oplus W'=V$. Since vertices in the spherical building associated to $GL(V)$ correspond to subspaces of $V$ and opposite vertices correspond to complementary subspaces this can be rephrased in terms of the building as follows.  

For a vector space $V$ over a field $k$, the group $GL(V)$ acts on a spherical building, call it $X$. For a subgroup $\Gamma$ of $GL(V)$, let $X^\Gamma$ be the set of points of $X$ which are fixed by the action of $\Gamma$, then $V$ is completely reducible if and only if every vertex of $X^\Gamma$ has an opposite vertex in $X^\Gamma$. This definition has an analogue in terms of parabolic subgroups containing $\Gamma$ since the simplices fixed by $\Gamma$ correspond to the parabolic subgroups containing $\Gamma$. Serre then extends the idea of complete reducibility to subgroups of any group which acts on a spherical building, specifically reductive algebraic groups.

The points fixed by $\Gamma$ form a convex subcomplex and the definition of complete reducibility can be applied to an arbitrary convex subcomplex of a spherical building. A convex subcomplex $Y$ is completely reducible if and only if every simplex of $Y$ has an opposite in $Y$.

In \cite{Ca09}, P. E. Caprace introduces the definition of completely reducible subgroups of a group $G$ with a twin $BN$-pair: a subgroup $H$ of $G$ is \emph{completely reducible} if $H$ is bounded and if given a parabolic subgroup $P$ of finite type which contains $H$, then there is a parabolic subgroup opposite $P$ which is of finite type and contains $H$.

A group $G$ with a twin $BN$-pair gives rise to a twin building $\C=(\C_+,\C_-)$ (see \cite{AB08} Chapter 8 for details) in such a way that the parabolic subgroups of $G$ correspond to the simplices (or equivalently, residues) of $\C$. Then the above definition of complete reducibility is equivalent to requiring that for every simplex (residue) in the fixed point subcomplex of $H$ in $\C$, there is an opposite simplex (residue) in the fixed point subcomplex of $H$.

The points fixed by $H$ form a convex subcomplex of $\C$ and we can extend this definition of complete reducibility to any convex subcomplex $Y=(Y_+,Y_-)$ of a twin building such that $Y_\epsilon$ is not empty and every simplex of $Y_\epsilon$ has an opposite simplex in $Y_{-\epsilon}$.

Convexity in a single building is more understood than convexity in twin buildings. P. Abramenko and K.S. Brown give a definition of convexity for chamber subcomplexes of a twin building in \cite{AB08} and Abramenko explores general convex subcomplexes in twin buildings in \cite{Ab96} but leaves several questions. Completely reducible subcomplexes are not always chamber subcomplexes so it is important to develop an understanding of general convex subcomplexes of twin buildings.

A subcomplex of a twin building is convex if and only if its intersection with any twin apartment is convex, so it suffices to study convexity in a twin apartment. A useful tool for studying apartments has been the Tits cone, which was introduced to study Coxeter complexes geometrically.  The Tits cone is a (possibly infinite) hyperplane arrangement of a subset of a real vector space and the chambers in an apartment correspond to simplicial cones defined by hyperplanes.  In nonspherical buildings the Tits cone is a convex subset of the vector space, so we can take the union of this subset with its negative and obtain a good representation of a twin apartment called the \emph{twin Tits cone}.

The definition of convexity in the vector space agrees with the definition of convexity in a building, but since the twin Tits cone is strictly contained in the vector space we need a slightly modified definition of convexity. We can define convexity in the twin Tits cone, $X$, as follows: if $X'$ is a subset of $X$ and $x,y$ are points in $X'$, then $X'$ is convex if and only if the geodesic $[x,y]\cap X$ is contained in $X'$. This leads to the following result about convexity in twin apartments.

\begin{thrm}Let $\Sigma'=(\Sigma'_+,\Sigma'_-)$ be a pair of nonempty subcomplexes of a twin apartment $\Sigma$ such that $\Sigma'_+$ and $\Sigma'_-$ each contain a spherical simplex. Then the following are equivalent:
  \begin{enumerate}
    \item $\Sigma'$ is convex in $\Sigma$, i.e. closed under projections.
    \item $\Sigma'$ is an intersection of twin roots.
    \item Let $X'$ be the union of the cells corresponding to $\Sigma'$ in the twin Tits cone $X$. Then $X'$ is convex in $X$.
  \end{enumerate}
\end{thrm}

Euclidean buildings have the unique property that there is an associated spherical building at infinity and in \cite{Ro03}, M. Ronan shows that for a twin Euclidean building there are sub-buildings of the corresponding buildings at infinity which are naturally twinned. Our main result allows us to only consider the subcomplexes of the spherical buildings at infinity to determine if a subcomplex is completely reducible.

\begin{thrm}[Main Theorem]Let $X=(X_+,X_-)$ be a Euclidean twin building and $Y=(Y_+,Y_-)$ a convex subcomplex of $X=(X_+,X_-)$. Let $I=(I_+,I_-)$ be the set of interior points in the buildings at infinity as in Section 4.2 and $Y^\infty=(Y^\infty_+,Y^\infty_-)$ the subcomplex of $I$ corresponding to $Y$. Then $Y$ is a completely reducible subcomplex of $X$ if and only if every simplex of maximal dimension in $Y^\infty$ has an interior opposite in $Y^\infty$.\end{thrm}

We also show that we only need to consider the set of vertices at infinity in our study of complete reducibility.

\begin{thrm} A convex subcomplex $Y$ is $X$-completely reducible if and only if every vertex in $Y^\infty$ has an interior opposite in $Y^\infty$.\end{thrm}

As an example for how this can be applied to a group with a twin $BN$-pair, Let $k$ be a field, $F=k(t)$, $R=k[t,t^{-1}]$, and $G=SL_n[R]$. Then $G$ has a twin $BN$-pair and an associated Euclidean twin building. Let $X=(X_+,X_-)$ be the geometric realization of this twin building and let $\Gamma$ be a subgroup of $G$ with fixed point complex $Y=(Y_+,Y_-)$ with $Y_\epsilon$ non empty for each $\epsilon\in\{+,-\}$. Then we have the following consequences of the preceding theorem.

\begin{propo}The subgroup $\Gamma$ is completely reducible if and only if every $\Gamma$-invariant $R$-submodule of $R^n$ which is a $R$ direct summand of $R^n$ has a $\Gamma$-invariant $R$-complement.\end{propo}

\begin{propo} Let $K=k(t)$ and let $\Gamma$ be a completely reducible subgroup of $G$. Then $R^n=M_1\oplus\cdots\oplus M_k$ where each $M_i$ is a $\Gamma$-invariant $R$ submodule such that $K\otimes_R M_i$ is irreducible in $K^n$.\end{propo}

\section{Background}

We assume the reader has a basic knowledge of buildings and we will briefly discuss the definition and some results that are useful here. The definitions and results in this chapter can also be found in \cite{AB08}. 

Let $(W,S)$ be a Coxeter system.

\begin{defs} A \emph{building of type} $(W,S)$ is a pair $(\mathcal{C},\delta)$ consisting of a nonempty set $\mathcal{C}$, of elements called \emph{chamber}, and a map
 $\delta:\mathcal{C}\times\mathcal{C}\rightarrow W$ called the \emph{Weyl-distance function}, such that for all $C,D\in \mathcal{C}$, the following conditions hold:
\begin{enumerate}
\item $\delta(C,D)=1$ if and only if $C=D$.
\item If $\delta(C,D)=w$ and $C'\in\mathcal{C}$ satisfies $\delta(C',C)=s\in S$ then $\delta(C',D)$ is $sw$ or $w$. If in addition $l(sw)=l(w)+1$, then $\delta(C',D)=sw$ where $l$ is the length function on $W$ with respect to $S$.
\item If $\delta(C,D)=w$ then for any $s\in S$ there is a chamber $C\in \mathcal{C}$ such that $\delta(C',C)=s$ and $\delta(C',D)=sw$.
\end{enumerate}\end{defs}

If $w=s_1s_2\cdots s_n$ in reduced form, then the \emph{length} of $w$ is $l(w)=n$. If $\delta(C,D)=w$, then the distance from $C$ to $D$ is $d(C,D):=l(w)$.

Let $J\subseteq S$ and let $W_J=\left<J\right>\leq W$. Two chambers $C,D$ in $\C$ are said to be $J$-equivalent if $\delta(C,D)\in W_J$. This is an equivalence relation and the equivalence classes are called $J$-residues. A subset $\mathcal{R}\subseteq \C$ is a \emph{residue} if it is a $J$-residue for some $J\subseteq S$ and $J$ is called the \emph{type} of $\mathcal{R}$, $S\setminus J$ is called the \emph{cotype} and $|J|$ is the \emph{rank}. A residue $\mathcal{R}$ is said to be \emph{spherical} if it is a $J$-residue for some $J$ such that $W_J$ is finite.

The above definition of a building is equivalent to the simplicial definition of a building (which is denoted by $\Delta$) and the residues of $\C$ correspond to the simplices of $\Delta$. The chambers of $\Delta$ correspond to the residues of type $\emptyset$ which are the chambers of $\C$, the simplices of codimension 1 (also called panels) correspond to the residues of type $\{s\}$ for $s\in S$, and the vertices correspond to residues of rank $|S|-1$.  In the simplicial building $\Delta$ we say that the type of a simplex is $S\setminus J$ where $J$ is the type of the corresponding residue, hence the type of a simplex in $\Delta$ is the cotype of the corresponding residue in  $\C$. So the vertices of $\Delta$ have type $\{s\}$ for $s\in S$ (note that each chamber of $\Delta$ contains exactly one vertex of type $\{s\}$ for each $s\in S$).

For $J\subseteq S$, every $J$-residues is isomorphic to a building of type $(W_J,J)$ and if $W_J$ is finite the $J$-residue and the corresponding simplex are said to be \emph{spherical}. 

An important property of spherical buildings is the existence of opposites. Let $\Sigma$ be an apartment of a spherical building of type $(W,S)$. Then there is a unique element of longest length in $W$, denoted $w_0$. If $C,C'$ are chambers of $\Sigma$ such that $\delta(C,C')=w_0$ then we say that $C$ and $C'$ are opposite. This induces an isometry on $\Sigma$ called the opposition involution which maps each chamber to its opposite in $\Sigma$. If $E$ is the geometric realization of $\Sigma$ then the opposition involution is defined on all the simplices of $E$, and for any simplex $A$ of $E$ the opposite of $A$ is $-A:=$op$_EA$. Note that if $A$ is a vertex of $E$ then $-A$ is the vertex which is diametrically opposite $A$.

We will work primarily with the simplicial building and its geometric realization but the Weyl distance definition best generalizes to twin buildings.

\subsection{Twin Buildings}

\begin{defs} A \emph{twin building} of type $(W,S)$ is a triple $(\mathcal{C}_+,\mathcal{C}_-,\delta^*)$ where $(\mathcal{C}_+,\delta_+)$ and $(\mathcal{C}_-,\delta_-)$ are buildings of type $(W,S)$ and $\delta^*:(\mathcal{C}_+\times\mathcal{C}_-) \cup (\mathcal{C}_-\times\mathcal{C}_+)$ is a \emph{codistance} function satisfying the following conditions for each $\epsilon\in\{+,-\}$, any $C\in\mathcal{C}_\epsilon$, and any $D\in\mathcal{C}_{-\epsilon}$ with $w:=\delta^*(C,D)$.
\begin{enumerate}
\item $\delta^*(C,D)=\delta^*(D,C)^{-1}$.
\item If $C'\in\mathcal{C}_\epsilon$ such that $\delta_{\epsilon}(C',C)=s\in S$ and $l(sw)<l(w)$ then $\delta^*(C',D)=sw$.
\item For any $s\in S$ there is a chamber $C'\in \mathcal{C}$ with $\delta_\epsilon(C',C)=s$ and $\delta^*(C',D)=sw$.
\end{enumerate}\end{defs}

For nonspherical buildings there is no element of maximal length so there is no notion of opposition, but in a twin building $\mathcal{C}=(\mathcal{C}_+,\mathcal{C}_-)$ we can say two chambers $C,D$ are \emph{opposite} if $\delta^*(C,D)=1$. We define the \emph{numerical codistance} between chambers by $d^*(C,D)=l(\delta^*(C,D))$. Then two chambers are opposite if and only if $d^*(C,D)=0$. 

\subsubsection{Projections and Convexity}

Assume that $\mathcal{C}=(\mathcal{C}_+,\mathcal{C}_-)$ is a twin building of type $(W,S)$. It is known that if $\mathcal{R}$ is a spherical residue of $\mathcal{C}_\epsilon$ and $D$ is a chamber of $\mathcal{C}_{-\epsilon}$ then there is a unique chamber $C_1\in \mathcal{R}$ such that $\delta^*(C_1,D)$ has maximal length in $\delta^*(\mathcal{R},D):=\{\delta^*(C,D)|C\in \mathcal{R}\}$. This chamber is called the \emph{projection} of $D$ onto $\mathcal{R}$ and is denoted by proj$_{\mathcal{R}}D$.  This chamber $C_1$ also satisfies the following equality for all $C\in \mathcal{R}$ \[\delta^*(C,D)=\delta_\epsilon(C,C_1)\delta^*(C_1,D)\] which gives the following analogue of the gate property:
\[d^*(C,D)=d^*(C_1,D)-d(C,C_1).\]  

Since residues correspond to simplices, the projection of a chamber $D\in \mathcal{C}_{-\epsilon}$  onto a spherical simplex $A\in\mathcal{C}_\epsilon$ is the unique chamber containing $A$ with maximal codistance from $D$.

A pair $(\mathcal{M}_+,\mathcal{M}_-)$ of nonempty subsets of $\mathcal{C}_+$ and $\mathcal{C}_-$ respectively is called \emph{convex} if proj$_\mathcal{P}C\in\mathcal{M}_+\cup\mathcal{M}_-$ for any $C\in\mathcal{M}_+\cup\mathcal{M}_-$ and any panel $\mathcal{P}$ that meets $\mathcal{M}_+\cup\mathcal{M}_-$. This is equivalent to saying that $(\mathcal{M}_+,\mathcal{M}_-)$ is closed under projections. Given two subsets $\mathcal{D}_1$ and $\mathcal{D}_2$ of $\C$, let Con$(\mathcal{D}_1,\mathcal{D}_2)$ denote the convex hull of $\mathcal{D}_1$ and $\mathcal{D}_2$. We will explore convexity in more detail in Chapter 3.

\subsubsection{Twin Apartments}

Consider a pair $(\Sigma_+,\Sigma_-)$ of nonempty subsets of a twin building $\mathcal{C}=(\mathcal{C}_+,\mathcal{C}_-)$ with $\Sigma_+$ an apartment of $\mathcal{C}_+$ and $\Sigma_-$ an apartment of $\mathcal{C}_-$, then $(\Sigma_+,\Sigma_-)$ is a \emph{twin apartment} of $\mathcal{C}$ if every chamber of $\Sigma_+\cup\Sigma_-$ is opposite to exactly one chamber of $\Sigma_+\cup\Sigma_-$. Then the \emph{opposition involution} op$_\Sigma$ associates to each chamber $C\in\Sigma_+\cup\Sigma_-$ its unique opposite in $\Sigma_+\cup\Sigma_-$. A twin apartment $\Sigma$ is the convex hull of any pair of opposite chambers contained in $\Sigma_+\cup\Sigma_-$ and such a pair $(C,C')$ of opposite chambers is called a \emph{fundamental pair} of chambers for $\Sigma$. The following lemma (5.173 in [AB08]) is useful throughout this paper.

\begin{lem}Let $\Sigma=(\Sigma_+,\Sigma_-)$ be a twin apartment and let $\epsilon=+$ or $-$.
\begin{enumerate}
\item op$_\Sigma:\Sigma_\epsilon\rightarrow\Sigma_{-\epsilon}$ is an isomorphism.
\item Given $C\in \Sigma_\epsilon$ and $D'\in\Sigma_{-\epsilon}$, let $D=$op$_\Sigma D'$. Then $\delta^*(C,D')=\delta_\epsilon(C,D)$.
\item Let $C,D,E$ be any chambers in $\Sigma_+\cup\Sigma_-$. Then $\delta(C,E)=\delta(C,D)\delta(D,E)$, where $\delta$ is the distance or codistance function which makes sense for each pair of chambers.
\item $\Sigma$ is convex in $\mathcal{C}$.
\end{enumerate}\end{lem}

\subsubsection{Twin Roots}

Given a twin apartment $\Sigma=(\Sigma_+,\Sigma_-)$ of a twin building $\mathcal{C}=(\C_+,\C_-)$, the pair $\alpha=(\alpha_+,\alpha_-)$ with $\alpha_\epsilon$ a root of $\Sigma_\epsilon$ for $\epsilon=\pm$ is a \emph{twin root} if op$_\Sigma(\alpha)=-\alpha=(-\alpha_+,-\alpha_-)$, where $-\alpha_\epsilon=\mbox{op}_\Sigma(\alpha_\epsilon)$. 

Consider a pair of adjacent chambers $C,D\in\Sigma_+$ and let $\alpha_+$ be the root of $\Sigma_+$ containing $C$ but not $D$. Let $C'=\mbox{op}_\Sigma C$ and $D'=\mbox{op}_\Sigma D$ (note that $C'$ and $D'$ are adjacent chambers of $\Sigma_-$) and let $\alpha_-$ be the root of $\Sigma_-$ containing $D'$ but not $C'$. Then $\alpha=(\alpha_+,\alpha_-)$ is a twin root of $\Sigma$ and is the convex hull of $C$ and $D'$. The following lemma (5.198 in \cite{AB08}) is very useful. Denote by $\mathcal{A}(\alpha)$ the set of apartments of $\C$ which contain $\alpha$.

\begin{lem}\label{twnewapart} Let $\alpha=(\alpha_+,\alpha_-)$ be a twin root, and let $\mathcal{P}$ be a panel in $\C_\epsilon$ which contains exactly one chamber $C$ of $\alpha$ for $\epsilon=\pm$. Then there is a bijection $\mathcal{P}\setminus\{C\}\rightarrow \mathcal{A}(\alpha)$ that assigns to each $D\in\mathcal{P}\setminus \{C\}$ the convex hull of $D$ and $\alpha$.\end{lem}

Given a simplex $A$ in a twin apartment $\Sigma$ we say that $A$ is a \emph{boundary simplex} of a twin root $\alpha\in\Sigma$ if there are chambers $C$ and $D$ having $A$ as a face such that $C\in\alpha$ and $D\not\in\alpha$. Then the above lemma says that if $P$ is a codimension 1 boundary simplex of a twin root $\alpha$ and if $D$ is any chamber not in $\alpha$ which has $P$ as a face, then there is a twin apartment containing $D$ and $\alpha$.

\subsection{Simplicial Approach}

Let $\C=(\C_+,\C_-)$ be a twin building, for $\epsilon=\pm$, let $\Delta_\epsilon$ be the simplicial building associated to $\C_\epsilon$, and $\Delta=(\Delta_+,\Delta_-)$. Let $X_\epsilon=|\Delta_\epsilon|$, the geometric realization of $\Delta_\epsilon$, and $X=(X_+,X_-)$. 

These are three equivalent views towards twin buildings and we will use the notations interchangeably throughout this paper.

\subsubsection{Sign Sequences}

Let $\Sigma$ be a Coxeter complex and let $\mathcal{H}$ be the complete set of walls of $\Sigma$. Each wall $H$ defines a pair of roots $\pm\alpha$ of $\Sigma$. Each simplex $A$ of $\Sigma$ is either in $+\alpha$, $-\alpha$ or $H$. We can assign a sign $\sigma_H(A)\in\{+,-,0\}$ where $\sigma_H(A)=0$ if and only if $A\in H$. The \emph{support} of $A$ is the intersection of walls $H$ such that $\sigma_H(A)=0$ (note that $A$ has the same dimension as its support, Proposition 3.99 in \cite{AB08}). The \emph{sign sequence} is defined as $\sigma(A)=\{\sigma_{H}(A)\}_{H\in\mathcal{H}}$.

Let $\Sigma=(\Sigma_+,\Sigma_-)$ be a twin apartment with geometric realization $E=(E_+,E_-)$. A \emph{twin wall} is a pair $H=(H_+,H_-)$ of walls in $E_+$ and $E_-$ respectively such that $H_-=$op$_\Sigma H_+$. If $\sigma_H(A)$ is the sign of a simplex $A$ with respect to the wall $H$ then $\sigma_H($op$_\Sigma A)=-\sigma_H(A)$.

\section{Convexity}

Convex subcomplexes of a single apartment are well understood and there are several equivalent definitions including being an intersection of roots, closed under products/projections, and closed under straight line segments in the corresponding Tits cone.  A subcomplex of building is \emph{convex} if its intersection with every apartment is convex in the apartment.

Convex subcomplexes of twin buildings are not as well understood.  There is one main definition in the literature to date, namely: a subcomplex of a twin building is \emph{convex} if it is closed under projections (within each building and between the two buildings). Proposition 5.193 of \cite{AB08} says that if the subcomplex contains a chamber then being closed under projections is equivalent to the subcomplex being an intersection of roots.  We show that this is also true if the subcomplex does not necessarily contain any chambers but does contain a sufficient number of spherical simplices.

\subsection{Projections}\begin{defs} Given simplices $A$ and $B$ of a building $\mathcal{C}$ the \emph{product}, $AB$, is defined as the simplex with sign sequence given by 
	\[ \sigma_H(AB) = \left\{ \begin{array}{ll}
         \sigma_H(A) & \mbox{if } \sigma_H(A)\neq 0;\\
        \sigma_H(B) & \mbox{if } \sigma_H(A)=0.\end{array} \right. \]  where $H$ ranges over the set of walls in an apartment containing $A$ and $B$. This product is also called the \emph{projection} of $B$ onto $A$ and denoted proj$_AB$. \end{defs}
        
\begin{defs} Given a twin building $\mathcal{C}=(\mathcal{C}_+,\mathcal{C}_-)$ let $A\in\mathcal{C}_{\epsilon}$ be a spherical simplex,  $B\in\mathcal{C}_{-\epsilon}$ be any simplex and $C\in\mathcal{C}_{-\epsilon}$ be any chamber containing $B$. Then proj$_AC$ is the unique chamber having $A$ as a face which has maximal codistance to $C$ and proj$_AB=\bigcap\mbox{proj}_AC$ where $C$ ranges over all chambers having $B$ as a face. \end{defs}

We can also characterize the projection of $B$ onto $A$ in a twin building in terms of sign sequences. We will need the following lemma. This is Proposition 4 in \cite{Ab96} and the proof uses the $W$-metric approach. We restate it in terms of simplices and give a simplicial proof. Note that $E_A$ is the link of $A$ which is the simplical building of the corresponding residue of $A$.

\begin{lem}\label{linkop} Let $E=(E_+,E_-)$ be a twin apartment and let $A\in E_\epsilon$ and $B\in E_{-\epsilon}$ be simplices  with $A$ spherical.  Let $E_A$ be the corresponding apartment in the link of $A$. Then \[\mbox{proj}_AB\;\;\mbox{op}_{E_A}\;\;\mbox{proj}_A(\mbox{op}_E B)\].\end{lem}
  \begin{proof} By definition, $\mbox{proj}_AB=\bigcap_{D\geq B}C$ where $D$ runs over all chambers having $B$ as a face and $C$ is a chamber such that $d^*(C,D)= \max\{d^*(C',D)| C'\geq A\}$ and $\mbox{proj}_A(\mbox{op}_E B)= \bigcap_{D\geq (\mbox{\tiny{op}}_E B)} C$ where $D$ runs over all chambers having $\mbox{op}_E B$ as a face and $C$ is a chamber such that $d(C,D)=\min\{d(C',D)|C'\geq A\}$. Note that $d^*(C',D)=d(C',\mbox{op}_E D)$.
  
  Let $D$ be a chamber having $B$ as a face and let $C_1=\mbox{proj}_AD$ and $C_2=\mbox{proj}_A(\mbox{op}_EB)$. Then $d^*(C_1,D)=d(C_1,\mbox{op}_ED)$ is maximal among distances $d^*(C,D)$ with $C\geq A$ and $d(C_2,\mbox{op}_ED)$ is minimal among distances $d(C,\mbox{op}_ED)$ with $C\geq A$. Hence $d(C_2,C_1)=d(C_1,\mbox{op}_ED)-d(C_2,\mbox{op}_E D)$ is maximal in the link of $A$. So $C_1$ is opposite $C_2$ in $L_A$.  Therefore, $\mbox{proj}_AB\;\;\mbox{op}_{E_A}\;\;\mbox{proj}_A(\mbox{op}_EB)$.
  \end{proof}

\begin{prop}\label{sign} Let $E=(E_+,E_-)$ be a twin apartment. Given simplices $A\in E_\epsilon$ and $B\in E_{-\epsilon}$ with $A$ spherical the sign sequence of proj$_AB$ is \[ \sigma_H(\mbox{proj}_AB) = \left\{ \begin{array}{ll}
         \sigma_H(A) & \mbox{if } \sigma_H(A)\neq 0;\\
        \sigma_H(B) & \mbox{if } \sigma_H(A)=0.\end{array} \right. \] where $H$ ranges over the twin walls of $E$. \end{prop}
        
\begin{proof} By Lemma \ref{linkop}, we know that proj$_AB=$op$_{E_A}(\mbox{proj}_A(\mbox{op}_EB)$. We also know that $\sigma_H(\mbox{op}_E B)=-\sigma_H(B)$. So we have the sign sequence
              \[\sigma_H(\mbox{proj}_A(\mbox{op}B))= \left\{ \begin{array}{ll}
                \sigma_H(A) & \mbox{if } \sigma_H(A)\neq 0;\\
                \sigma_H(\mbox{op $B$}) & \mbox{if } \sigma_H(A)=0\end{array} \right. \]
              \[ = \left\{ \begin{array}{ll}
                \sigma_H(A) & \mbox{if }\sigma_H(A)\neq 0;\\
                -\sigma_H(B) & \mbox{if }\sigma_H(A)=0. \end{array} \right. \]
                
Since the walls of $E_A$ correspond bijectively to the walls of $E$ containing $A$, the opposition involution op$_{E_A}$ negates only the signs corresponding to the walls containing $A$. Therefore, 
							\[\sigma_H(\mbox{proj}_AB) = \left\{\begin{array}{ll}
							  \sigma_H(A) & \mbox{if $\sigma_H(A)\neq 0$};\\
                \sigma_H(B) & \mbox{if $\sigma_H(A)=0$}.\end{array} \right. \]
                
\end{proof}

\subsection{Twin Tits cone}

Let $\Sigma=(\Sigma_+,\Sigma_-)$ be a twin apartment of type $(W,S)$, where $W$ is infinite and irreducible. The chambers of $\Sigma_+$ correspond to simplicial cones in a real vector space and the union of these cells $X_+$ is called the \emph{Tits cone} of $\Sigma_+$ as in section 2.6 of \cite{AB08}.  The subset $X_+$ of $V$ is a convex subset of $V$ and since $W$ is infinite $X_+\neq V$.  Let $X_-=-X_+$. So $X_-$ is a Tits cone representation of $\Sigma_-$ and $X_+\cap X_-=0$. We define the \emph{twin Tits cone} as $X=X_+\cup X_-$. Let $C$ be a chamber of $\Sigma_+$, and abusing notation also the corresponding simplicial cone in $X_+$. The simplicial cone $-C$ corresponds to the chamber of $\Sigma_+$ which is opposite $C$. Then in $X$, $-wC=w\mbox{op}_{\Sigma}C$ and two chambers, $D$ and $D'$, in $X$ are said to be \emph{opposite in $X$} if $D'=-D$. 

\begin{figure}\begin{center}\includegraphics[scale=.4]{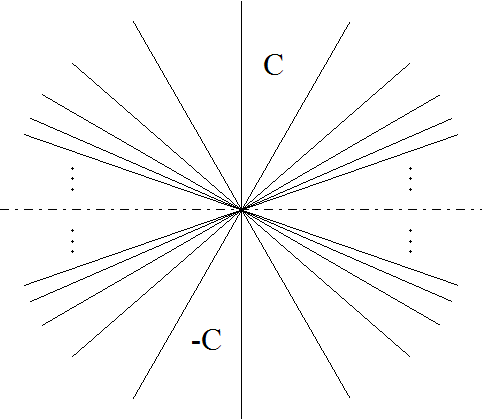}\caption{\label{titscone}The twin Tits cone for $D_\infty$}\end{center}\end{figure}

\begin{ex} Let $W=D_\infty=\left<s,t|s^2=t^2=1\right>$. The Tits cone corresponding to $W$ is the open upper half plane of $\R^2$ plus the origin and the twin Tits cone is $\R^2$ not inculding the ponits $(x,0)$ for $x\neq 0$ as in Figure \ref{titscone}. \end{ex}

\begin{prop} Two chambers $D\in \Sigma_+$ and $D'\in\Sigma_-$ are opposite in $\Sigma$ if and only if their corresponding chambers in $X$ are opposite.\end{prop}
  \begin{proof} Let $(C,C')$ be a fundamental pair of $\Sigma$, and abusing notation, also the fundamental pair of the twin Tits cone $X$. Assume that $D$ and $D'$ are opposite in $\Sigma$, hence $\delta^*(D,D')=1$. From Lemma 5.173 in \cite{AB08}, we get \[\delta_+(C,D)=\delta^*(C,C')\delta_-(C',D')\delta^*(D',D)=\delta_-(C',D').\] Let $w=\delta_+(C,D)=\delta_-(C',D')$. Hence in the Tits cone $D'=wC'=-wC=-D$ and so $D$ and $D'$ are opposite in $X$.
  
  Conversely, suppose $D'=-D$ in $X$. Let $w=\delta_+(C,D)$, then in $X$, $D=wC$ so $D'=-wC=wC'$. Then we have that $\delta_-(C',D')=w$ hence \[\delta^*(D,D')=\delta_+(D,C)\delta^*(C,C')=\delta_-(C',D').\] So $\delta^*(D,D')=w^{-1}w=1$.
  \end{proof}
  
Since the twin Tits cone $X$ is not convex in $V$, a slightly different definition of a convex subset of $X$ is needed.

\begin{defs} A union of cells $Y$ contained in $X$ is \emph{convex} in $X$ if for any two points $x,y\in Y$, $[x,y]\cap X\subseteq Y$, where $[x,y]$ is the straight line connecting $x$ and $y$ in $V$. \end{defs}

\begin{figure}\begin{center}\includegraphics[scale=.4]{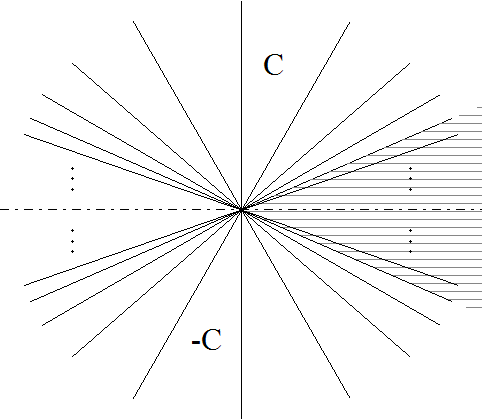}\caption{\label{titscone2}A convex subset of the twin Tits cone for $D_\infty$}\end{center}\end{figure}

\begin{ex}For $W=D_\infty=\left<s,t|s^2=t^2=1\right>$ with twin Tits cone $X$. The shaded region minus the dotted line in Figure \ref{titscone2}is convex in $X$.\end{ex}

\subsection{Convexity in a Twin Apartment}

Given a simplicial complex $\Delta$ of finite dimension, we say that $\Delta$ is a \emph{chamber complex} if all maximal simplices have the same dimension and can be connected by a gallery. Any building, and any apartment in a building is a chamber complex. Also, any convex subcomplex $\Sigma'$ of an apartment $\Sigma$ is a chamber complex (though the chambers of $\Sigma'$ may not be chambers of $\Sigma$)(Proposition 3.136 of \cite{AB08}).

We need the following lemma which guarantees a certain number of spherical simplices given at least one of maximal dimension.

\begin{lem}\label{sphcha} Let $\Sigma$ be a Coxeter complex and let $\Sigma'$ be a convex subcomplex of $\Sigma$ which contains at least one spherical simplex. Then every $\Sigma'$-chamber is spherical.\end{lem}
  \begin{proof} Since $\Sigma'$ is convex and contains a spherical simplex $C$, it must contain a spherical $\Sigma'$-chamber $A$ which has $C$ as a face. Now assume that there is a $\Sigma'$-chamber $B$ which is not spherical and consider $BA=$proj$_BA$. Since $A$ and $B$ both have maximal dimension in $\Sigma'$ and $\Sigma'$ is closed under projections, we must have that $BA$ has maximal dimension in $\Sigma'$, hence $BA=B$. Consider the sign sequence definition of projection:
  \[ \sigma_H(BA) = \left\{ \begin{array}{ll}
         \sigma_H(B) & \mbox{if $\sigma_H(B)\neq 0$};\\
        \sigma_H(A) & \mbox{if $\sigma_H(B)=0$}.\end{array} \right. \] 
Then $BA$ is spherical if and only if $\sigma_H(BA)=0$ for finitely many $H$.  Since $A$ is spherical, $\sigma_H(A)=0$ for finitely many $H$, hence $\sigma_H(BA)=0$ for finitely many $H$, so $BA=B$ is spherical which is a contradiction, so $B$ does not have maximal dimension in $\Sigma'$.\end{proof}

This brings us to our main result, giving several equivalent definitions of convexity in a twin apartment.

\begin{thm}\label{convex}Let $\Sigma'=(\Sigma'_+,\Sigma'_-)$ be a pair of nonempty subcomplexes of a twin apartment $\Sigma$ such that $\Sigma'_+$ and $\Sigma'_-$ each contain a spherical simplex. Then the following are equivalent:
  \begin{enumerate}
    \item $\Sigma'$ is convex in $\Sigma$, i.e. closed under projections.
    \item $\Sigma'$ is an intersection of twin roots.
    \item Let $X'$ be the union of the cells corresponding to $\Sigma'$ in the twin Tits cone $X$. Then $X'$ is convex in $X$.
  \end{enumerate}
\end{thm}

\begin{proof} We will prove $(1)\Rightarrow (2) \Rightarrow (3) \Rightarrow (1)$.
    \begin{enumerate}
      \item[$(1) \Rightarrow(2)$] 
      	Let $\mathcal{S}$ be the support of $\Sigma'_\epsilon$. Then $\mathcal{S}$ is a convex subcomplex of $\Sigma_\epsilon$ containing at least one spherical simplex and by Lemma \ref{sphcha}, all $\mathcal{S}$-chambers are also spherical. Also, all $\Sigma'_\epsilon$-chambers are spherical.      	
      
      	We know from Lemma 3.137 in [1] that $\Sigma'_{\epsilon}$ is an intersection of roots $\alpha_\epsilon$, each defined by a boundary panel of $\Sigma'_\epsilon$. This boundary panel, as defined in the proof of the Lemma, is the face of exactly two spherical $\mathcal{S}$-chambers, so it is spherical.  We need to show that $\Sigma'_{-\epsilon}\subseteq\alpha_{-\epsilon}$. We will use contradiction.
          
          Let $B$ be a simplex of $\Sigma'_{-\epsilon}$ and $A$ a boundary panel in  $\Sigma'_{\epsilon}\cap\partial\alpha_\epsilon$. Note that proj$_AB=\bigcap_{D\geq B}\mbox{proj}_AD$. So let $D$ be any chamber of $\Sigma$ having $B$ as a face and let $C=\mbox{proj}_AD$.
          
          Now assume that $B\not\in\alpha_{-\epsilon}$. Then op $B\in\alpha_{\epsilon}\setminus\partial\alpha_\epsilon$ and op $D\in \alpha_\epsilon$. Hence $C=\mbox{proj}_AD\in -\alpha_{\epsilon}$. Since this holds for all chambers having $B$ as a face we must have that proj$_AB\in -\alpha_\epsilon\setminus\partial\alpha_\epsilon$ which is a contradiction. Therefore, $\Sigma'$ is the intersection of twin roots with $\Sigma'\subset\alpha$ and $\Sigma'\cap\partial\alpha\neq\emptyset$.
          
      \item[$(2)\Rightarrow (3)$] It is enough to show that twin roots in $\Sigma$ correspond to half-spaces in $X$. Then an intersection of twin roots in $\Sigma$ corresponds to an intersection of half-spaces in $X$, which is a convex set. To show this, note that roots in $X_+$ correspond to roots in $\Sigma_+$. So for a given $\alpha_+\subset\Sigma_+$ and corresponding $(\alpha_X)_+\subset X_+$ it suffices to show that $\alpha_-\subset\Sigma_-$corresponds to $(\alpha_X)_-\subset X_-$. This follows from the fact that opposition is preserved:
          \[\alpha_-=\mbox{op}_\Sigma(-\alpha_+)\leftrightarrow\mbox{op}_X(-(\alpha_X)_+)=(\alpha_X)_-.\]
          
      \item[$(3)\Rightarrow (1)$] Given $A,B\in \Sigma'$ with $A$ spherical, we want to show proj$_AB\in \Sigma'$. We may assume $A\in\Sigma'_+$ and $B\in\Sigma'_-$. Let $x$ be a point in the interior of $A$ and $y$ a point in the interior of $B$, with $A$ and $B$ viewed as cells of $X$. Let $y'=\mbox{op}_Xy$.
          
          Let $l_1$ be the segment of the line $[x,y]$ starting at $x$ and having length $\epsilon$. Let $C$ be the cell of minimal dimension containing $l_1$. We claim that proj$_AB=C$. Then since $X'$ is convex, any cell meeting [x,y] in its interior is in $X'$. Hence proj$_AB$ is in $\Sigma'$.
          
          To prove the claim, first note that $D=\mbox{proj}_A(\mbox{op }B)$ corresponds to the cell containing a segment of $[y',x]$ starting at $x$ and having length $\epsilon$. In the link of $A$, the cell opposite $D$ corresponds to the cell containing an extension of $[y',x]$ starting at $x$ and having length $\epsilon$; call this extension $l_2$. Let $C'$ be the cell of minimal dimension containing $l_2$. By Lemma \ref{linkop}, $C'=\mbox{proj}_AB$. It remains to show that $C=C'$. This amounts to showing that $l_1$ and $l_2$ are not separated by any hyperplane of $X$. For any hyperplane $H$ of $X$ there are three cases to consider: $A\not\in H$, $A\in H$ and $B\not\in H$, and $A,B\in H$.
          
          First, assume $A\not\in H$. Then there is some positive distance between $x$ and $H$. Since $\epsilon$ is arbitrarily small, $l_1$ and $l_2$ are not separated by $H$. Second, assume that $A\in H$ and $B\not\in H$. Then by definition, $l_1$ is on the same side of $H$ as $B$ and $l_2$ is on the opposite side of $H$ as op $B$. Hence $l_2$ is on the same side of $H$ as $B$ and $l_1$. Thirdly, assume $A,B\in H$. Then op $B\in H$ so $l_1$ and $l_2$ are in $H$.

    \end{enumerate}
  \end{proof}
  
\subsection{``Coconvexity''}

In \cite{Ab96}, P. Abramenko discusses a notion of ``coconvexity'' which is defined as closure under projections, but only those projections between the two components of the twin building, not projections within each component. In that book P. Abramenko states without proof the following proposition which we prove here.

\begin{prop} Let $A\in\Sigma_+$, $B\in\Sigma_-$ be spherical simplices. Then the coconvex hull of $A$ and $B$, Con$^*(A,B)$, is the intersection of all twin roots containing $A$ and $B$.\end{prop}

\begin{proof} Since Con$^*(A,B)$ is contained in the convex hull of $A$ and $B$, Proposition \ref{convex} gives the inclusion $\mbox{Con}^*(A,B)\subseteq \bigcap\{\alpha|\,A,B\in\alpha\}$. Note that twin roots of $\Sigma$ are in one to one correspondence with half-spaces of the twin Tits cone $X$ so we need to show that $\mbox{Con}^*_X(A,B)\supseteq \bigcap\{\mbox{half-spaces containing $A$ and $B$}\}$.  Let $\mathcal{D}$ be the intersection of hyperplanes containing $A$ and $B$. Note that dim$(\mathcal{D})=$dim$(\mbox{Con}^*(A,B))$: since $\mbox{Con}^*(A,B)$ is contained in $ \mathcal{D}$ we know that dim$(\mbox{Con}^*(A,B))\leq $dim$(\mathcal{D})$, and by the sign sequence of proj$^*_AB$ we know that the hyperplanes containing proj$^*_AB$ are exactly those containing both $A$ and $B$ so that dim$(\mbox{proj}_A^*B)=$dim$(\D)$ hence dim$(\mathcal{D})\leq$ dim$(\mbox{Con}^*(A,B))$. Since $X_+$ contains an infinite hyperplane arrangement and $\mathcal{D}_+$ is a convex subcomplex of $X_+$, all the results in [\cite{AB08}, section 2.7] apply to $\mathcal{D}_+$. So for the remainder of the proof, we will be working in $\mathcal{D}$, so by ''chamber'' we will mean $\mathcal{D}$-chamber, etc.

Let $\mathcal{D}_0$ be the intersection of half-spaces of $X_+$ containing $A$ but not op $B$, which is the intersection of $X_+$ with the intersection of the half-spaces of $X$ containing $A$ and $B$. Note that there is only one chamber of $\D_0$ having $A$ as a face: any two chambers containing $A$ are separated by at least one hyperplane $H$, one of these chambers would have to be on the same side of $H$ as op $B$ and therefore would not be in $\D_0$. Consider the sign sequence of $C_0$: since $C_0$ has $A$ as a face, if $\sigma_H(A)\neq 0$ then $\sigma_H(C_0)=\sigma_H(A)$, and since all the hyperplanes containing $A$ separate $C_0$ from op$B$, if $\sigma_H(A)=0$ then $\sigma_H(C_0)=-\sigma_H(\mbox{op } B)$ which is the same sign sequence as proj$_A^*B$ from Lemma \ref{sign}. Hence, $C_0=$proj$_A^*B$.

Now let $C_1$ be in $\D_0$ with distance 1 from $C_0$. Let $D_1=C_0\cap C_1$. Since $C_1\in\D_0$, $|\Se(\mbox{op} B, C_1)|\geq |\Se(\mbox{op }B,C_0)|$, and since $C_0$ is the only chamber of $\D_0$ containing $A$, we have strict inequality. Hence the hyperplane defined by $D_1$ separates op $B$ from $C_1$ so $C_1\in\D_1$ which is defined to be the intersection of half-spaces containing $D_1$ and $B$, and $C_1$ is the only chamber of $\D_1$ containing $D_1$. By the above argument $C_1=$proj$^*_{D_1}B$, hence $C_1\in\mbox{Con}^*(A,B)$.

We continue by inducting on the distance from $C_0$. Assume that all chambers of distance less than $n$ from $C_0$ are in $\mbox{Con}^*(A,B)$. Let $C_n$ be a chamber of $\D_0$ of distance $n$ from $C_0$. Then $C_n$ is adjacent to a chamber $C_{n-1}$ which is in $\mbox{Con}^*(A,B)$ and $C_{n-1}=$proj$^*_{D_{n-1}}B$ for some $D_{n-1}$. Let $\D_{n-1}$ be the intersection of halfspaces containing $D_{n-1}$ and $B$. If, in the above proof that $C_1\in\mbox{Con}^*(A,B)$, we make the following identifications: \[\D_{n-1}\longrightarrow \D_0\] \[ D_{n-1}\longrightarrow A\] \[C_{n-1}\longrightarrow C_0\] \[C_n\longrightarrow C_1\] we get $C_n\in\mbox{Con}^*(D_{n-1},B)\subset \mbox{Con}^*(A,B)$ because $D_{n-1}\in\mbox{Con}^*(A,B)$.

\end{proof}

The next example shows that being closed only under projections between the two components does not guarantee convexity in each component of the twin building, leading to the conclusion that we need to require projections within each component in our definition of a convex subcomplex.

\begin{center}\begin{figure}\includegraphics[scale=.5]{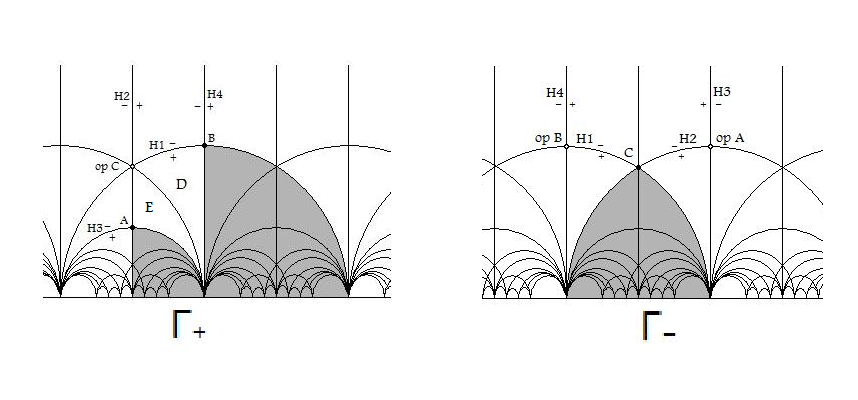}\caption{\label{example}A ``coconvex'' subcomplex which is not convex.}\end{figure}\end{center} 

\begin{ex} Consider the group $W=\left<u,v,w|u^2=v^2 = w^2 = (uv)^3 = (uw)^2\right>$ with generating set $S=\{u,v,w\}$. Then the hyperbolic planes, $\Sigma=(\Sigma_+,\Sigma_-)$ in Figure \ref{example} form a thin twin building of type $(W,S)$. Let $C$ be a vertex of $\Sigma_-$ of type $\{w\}$ and let $H_1$ and $H_2$ be walls containing $C$. Let $B$ be a spherical simplex of $\Sigma_+$ of type $\{v\}$ which is in $H_1$ such that $B$ and op$_\Sigma C$ are vertices of a common chamber. Similarly, let $A$ be a spherical simplex of $\Sigma_+$ of type $\{v\}$ which is in $H_2$ such that $A$ and op$_\Sigma C$ are vertices of a common chamber. Let $H_3$ be the wall containing $A$ but not op$_\Sigma C$ and let $H_4$ be the wall containing $B$ but not op$_\Sigma C$ (since $A$ and $B$ have type $\{v\}$, their links are isomorphic to a Coxeter complex of type $(W_J,J)$ where $J=\{u,w\}$ and $W_J=\left<u,w|u^2=w^2=(uw)^2=1\right>$ which has exactly two walls). 

The coconvex hull of $A$, $B$, and $C$ is the shaded subcomplex $\Gamma=(\Gamma_+,\Gamma_-)$ in the Figure \ref{example}. Since $\sigma_1(C)=\sigma_1(B)=0$ and $\sigma_1(A)=+$ we know that $\sigma_1(s)\geq 0$ for all $s\in\Gamma$ and similarly $\sigma_2(s)\geq 0$ for all $s\in \Gamma$. Since the roots defined by $H_1$ and $H_3$ are nested and similarly for the roots defined by $H_2$ and $H_4$ we have that $\sigma_3(s)\geq 0$ and $\sigma_4(s)\geq 0$ for all $s\in\Gamma_-$.  The sign sequence for $B$ with respect to these hyperplanes is $\{0\;+\;-\;0\}$ and for $A$ we have $\{+\;0\;0\;-\}$. From what we just said about $\Gamma_-$ we know that these zeros can only be replaced with $+$ hence there is no way to get the sign sequence $\{+\;+\;-\;-\}$ which is the sign sequence for both $D$ and $E$.  Therefore, neither $D$ nor $E$ is in $\Gamma$ and $\Gamma_+$ is not convex.
\end{ex}

\section{Twin Buildings at Infinity}

\subsection{A Single Building at Infinity}

To every Euclidean building we can associate a spherical building by attaching a sphere at infinity to each apartment. This is achieved as follows (see chapter 11 of \cite{AB08}).

Let $E$ be the geometric realization of a Euclidean Coxeter complex of type $(W,S)$ with $\mathcal{H}$ the corresponding set of hyperplanes in $E$. Let $x$ be a point of $E$ and $\overline{\mathcal{H}}$ be the set of hyperplanes through $x$ which are parallel to some element in $\mathcal{H}$. Then $\overline{\mathcal{H}}$ defines a decomposition of $E$ into conical cells, called \emph{conical cells based at x}. If $x$ is a special vertex (every hyperplane of $\mathcal{H}$ is parallel to a hyperplane of $\overline{\mathcal{H}}$) then $\overline{\mathcal{H}}$ is a subset of $\mathcal{H}$ and is isomorphic to the set of hyperplanes corresponding to a Coxeter complex of type $(\overline{W},S)$ where $\overline{W}$ is the finite reflection group consisting of the linear parts of the elements of $W$. 

Let $\mathfrak{D}$ be a cell associated to $\overline{W}$, then for any point $y\in E$ the conical cells based at $E$ are the translates $\mathfrak{A}=y+\mathfrak{D}$, and if $\mathfrak{D}$ is a chamber, then $\mathfrak{A}$ is called a \emph{sector}. Figure \ref{sector} shows a sector based at a vertex $y$. The bold lines in the figure, which are called rays, are also conical cells based at $y$. The vertex $x$ is a special vertex and $y$ is not a special vertex. 

\begin{figure}\begin{center}\includegraphics[scale=.4]{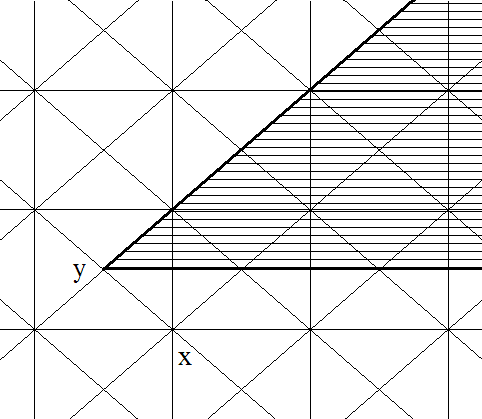}\caption{\label{sector}A sector based at vertex $x$.}\end{center}\end{figure}

Let $X$ be the geometric realization of a Euclidean building of type $(W,S)$. Then the building at infinity, $X^\infty$, is the collection of ends of parallel classes of rays.  The simplices of $X^\infty$ correspond to parallel classes of conical cells and the chambers of $X^\infty$ correspond to parallel classes of sectors.  Two conical cells are parallel if the distance between them is bounded. For sectors this implies that their intersection contains a sector. A sector $\mathfrak{C}'\subseteq\mathfrak{C}$ is called a \emph{subsector} of $\mathfrak{C}$. Note that $X^\infty$ is a spherical building of type $(\overline{W},S)$.

Let $\A=x+\mathfrak{D}$ be a conical cell based at $x$ with direction $\mathfrak{D}$ in an apartment $E$. Let $\mathfrak{D}'$ be the cell associated to $\overline{W}$ which is opposite $\mathfrak{D}$. Define the \emph{reversal} of $\A$ in $E$ as rev$_E\A:=x+\mathfrak{D}'$.  This is equivalent to the definition given in \cite{Ro03}, where rev$_E\A$ is defined as the image of $\A$ under the isometry sending each point of $\A$ to the point diametrically opposite to it with respect to the base point $x$. 

The following lemma is a generalization of exercise 11.50 in \cite{AB08}.

\begin{lem}\label{rootcon} Let $H$ be a wall, $\mathfrak{A}$ a conical cell in an apartment $E$. Then one of the roots of $E$ determined by $H$ contains a conical cell $\mathfrak{A}'$ such that $\mathfrak{A}'\subseteq\mathfrak{A}$ has the same direction as $\mathfrak{A}$.\end{lem}

\begin{proof} Let $\mathfrak{D}$ be the direction of $\mathfrak{A}$, and let the wall $H$ be determined by an equation $f=c$.  We may assume that $f\geq 0$ on $\mathfrak{D}$ so that $f(x)\geq c$ for some $x\in\mathfrak{A}$. Then the conical cell $\mathfrak{A}'=x+\mathfrak{D}$ of $\mathfrak{A}$ is contained in the root determined by $f\geq c$. \end{proof}

\subsection{Twin Buildings at Infinity}

Now consider a Euclidean twin building $X=(X_+,X_-)$. 

\subsubsection{Conical Cells and Twin Apartments}

Let $E=(E_+,E_-)$ be a twin apartment, and $\mathfrak{A}=x+\mathfrak{D}$ a conical cell based at $x$ with direction $\mathfrak{D}$ in $E_\epsilon$ for $\epsilon=+$ or $-$. Then op$_E\A$ is a conical cell based at op$_E x$ and the \emph{twin} of $\A$ in $E$, tw$_E\A$, is the reversal of the opposite of $\A$. So tw$_E\A=$rev$_E($op$_E\A)=$op$_E($rev$_E\A)$ (see Figure \ref{twinsector}) . Note that if $\A$ is a sector twinned with $\A'$ and $\mathfrak{C}$ is any sector containing $\A$ then $\A'$ contains a sector $\mathfrak{C}'$ twinned with $\mathfrak{C}$.

The following are generalizations of Proposition 11.62 and Theorem 11.63(1) of \cite{AB08} for twin apartments and general conical cells.

\begin{figure}\begin{center}\includegraphics[scale=.5]{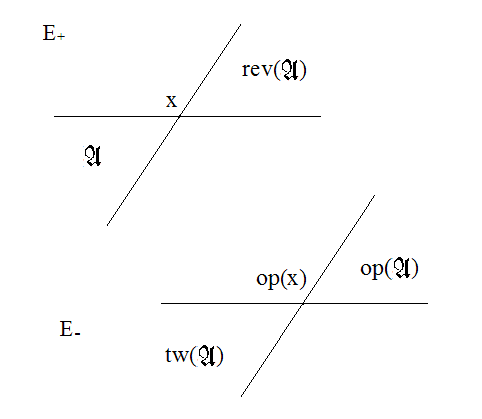}\caption{\label{twinsector}Twin sectors in a twin apartment}\end{center}\end{figure}

\begin{prop}If $\A$ is a conical cell of a twin apartment $E$, then $\A$ is a conical cell of every twin apartment containing it.
\end{prop}

\begin{proof} This proof is the same as that for Proposition 11.62 in \cite{AB08}.\end{proof}

\begin{prop}\label{comapart}Given a conical cell $\A=x+\mathfrak{D}$ in a twin apartment $E_0$ and a simplex $A$ in $X$, there is a twin apartment containing $A$ and a conical cell $\A'\subseteq\A$ having the same direction as $\A$.\end{prop}

\begin{proof} Let $E_0$ be a twin apartment containing $\A$. Consider a minimal gallery from $A$ to $E_0$, $\Gamma:A\leq C_0,\ldots, C_{n-1},C_n$ where all chambers of $\Gamma$ are not in $E_0$ except $C_n$. The chambers $C_{n-1}$ and $C_n$ are in a panel $\mathcal{P}$, defining a wall $H$ of $E_0$.  By Lemma \ref{rootcon}, one of the roots of $E_0$ defined by $H$ contains a conical cell $\A_1=x_1+\mathfrak{D}$ contained in $\A$, call it $\alpha$.  Then $\mathcal{P}$ intersects $\alpha$ in a single chamber $C$, and since $C_{n-1}\not\in E_0$, we have $C_{n-1}\in\mathcal{P}\setminus \{C\}$. If $\A$ and $A$ are both in $X_\epsilon$ then by Exercise 5.83 in [AB08], there is an apartment $E_1$ containing $\alpha$ and $C_{n-1}$. If $\A$ and $A$ are not both in $X_\epsilon$, then by Lemma 5.198 in [AB08] we have the same conclusion. So $E_1$ contains $\A_1$ and $C_{n-1}$. 

In the gallery $\Gamma$, there is an $m<n$ such that $C_m\in E_1$ and $C_{m-1}\not\in E_1$. We can argue as above to find an apartment $E_2$ containing a conical cell $\A_2=x_2+\mathfrak{D}$ of $\A$ and $C_{m-1}$. Continue to construct apartments $E_i$ in this way. Since the distance from $E_{i+1}$ to $A$ is strictly less than the distance from $E_i$ to $A$, we will find an apartment $E$ containing a conical cell $\A'=x'+\mathfrak{D}\subseteq\A$ and $A$ in at most $n$ iterations of this construction. \end{proof}

\subsubsection{Interior Sub-buildings at Infinity}

For a twin Euclidean building $X=(X_+,X_-)$ there are corresponding spherical buildings at infinity $(X_+)^\infty$ and $(X_-)^\infty$. Following \cite{Ro03}, the sectors that lie in a twin apartment of $X$ are called \emph{interior}, and if two sectors are parallel and one is interior so is the other. The chambers of $(X_\pm)^\infty$ that are parallel classes of interior sectors are called \emph{interior chambers} and if $E=(E_+,E_-)$ is a twin apartment, then $(E_+)^\infty$ and $(E_-)^\infty$ are \emph{interior apartments}.  The subcomplexes of $(X_\pm)^\infty$ consisting of interior chambers will be denoted $I_\pm$.   

It is important to note that not every apartment of $I_\epsilon$ is an interior apartment. For example, consider the case where $X=(X_+,X_-)$ is a twin tree. Then $I_\epsilon$ is a disjoint set of points for each $\epsilon=+$ or $-$ and any pair $x,y\in I_\epsilon$ forms an apartment in $I_\epsilon$ and it is known that not every apartment in $X_\epsilon$ is part of a twin apartment in $X$.  

In \cite{Ro03}, it is shown that $I_+$ and $I_-$ are sub-buildings of $(X_+)^\infty$ and $(X_-)^\infty$ and will be called \emph{interior sub-buildings at infinity}. The following are results in \cite{Ro03} which imply that the twinning of sectors mentioned in Section 4.2.1 induces a canonical isomorphism between $I_+$ and $I_-$.

\begin{prop} Let $\mathfrak{C}$ be a sector twinned with sectors $\mathfrak{C}_1$ and $\mathfrak{C}_2$. Then $\mathfrak{C}_1$ and $\mathfrak{C}_2$ are parallel.\end{prop}

\begin{coro} Let $\mathfrak{C}_1$ and $\mathfrak{C}_2$ be parallel sectors, twinned with $\A_1$ and $\A_2$ respectively. Then $\A_1$ and $\A_2$ are parallel.\end{coro}

Let $A$ and $A'$ be simplices of $I_+$ and $I_-$ respectively. Then $A$ and $A'$ correspond to classes of parallel interior conical cells $[\mathfrak{A}]$ and $[\mathfrak{A}']$ respectively. We say that $A$ and $A'$ are \emph{interior opposite} if and only if there exist conical cells $\mathfrak{U}\in [\mathfrak{A}]$ and $\mathfrak{U}'\in [\mathfrak{A}']$ such that $\mathfrak{U}$ and $\mathfrak{U}'$ are opposite conical cells in a twin apartment. This is equivalent to saying that $A$ and $A'$ are opposite in an interior apartment. It is important to note that interior opposition is a stronger condition than opposition in the spherical building $I_\pm$ as can be seen in the case of a twin tree.

\section{Complete Reducibility}

Let $X$ be the geometric realization of a spherical building. In \cite{Se04}, Serre defines the notion of complete reducibility for a convex subcomplex of $X$ and gives equivalent criteria to determine if a convex subcomplex is completely reducible.

\begin{defs} A convex subcomplex $Y$ of $X$ is said to be \emph{completely reducible} if for every point $y\in Y$, there exists a point $y'\in Y$ such that $y$ is opposite $y'$, or equivalently for every simplex $s$ of $Y$ there exists an opposite simplex $s'$ of $Y$.\end{defs}

\begin{thm}\label{Serthm}[Theorem 2.1, \cite{Se04}] Let $Y$ be a convex subcomplex of $X$. Then the following are equivalent:
\begin{enumerate}
\item[(a)] $Y$ is completely reducible in $X$.
\item[(b)] $Y$ contains a pair of opposite simplices which have the same dimension as $Y$.
\item[(c)] $Y$ contains a Levi sphere of the same dimension as $Y$.
\item[(d)] $Y$ is not contractible.
\item[(e)] For every vertex of $Y$, $Y$ contains an opposite vertex.
\end{enumerate}
\end{thm}

A Levi sphere $S$ of $X$ is a subcomplex of an apartment $E$ of $X$ which is the convex hull of a pair of opposite simplices, $(s,s')$. Note that $S$ is the support of $s$, which is the intersection of walls containing $s$. If $E\cong S^2$, then the Levi spheres are $E$ itself, any subcomplex which is a great circle, and any pair of opposite vertices. 

In Serre's proof of this theorem, he shows (c) implies (d) implies (a). Since this argument does not generalize to twin buildings we give a direct proof of (c) implies (a).

\begin{prop}\label{Levisph}  Let $Y$ be a convex subcomplex of a spherical building $X$. If $Y$ contains a Levi sphere, $S$, with $\dim(S)=\dim(Y)$ then $Y$ is completely reducible.\end{prop}
  \begin{proof} Let $A$ be any simplex of $Y$ not in $S$. We must show that $A$ has an opposite in $Y$. Note that we only need to consider simplices $A$ with $\dim(A)=\dim(Y)$. We can induct on the distance from $S$ to $A$ using the fact that $Y$ is convex, to reduce to the case when $A$ is adjacent to $S$ (i.e. $A$ has as a face a simplex $x\in S$ with $\dim(x)=\dim(S)-1$). Let $x$ be such a simplex and let $y$ be the simplex in $S$ which is opposite $x$ and let $B$ be one of the two simplices of dimension $\dim(S)$ in $S$ containing $y$. Consider the convex hull of $x$ and $B$, $\mbox{Con}(x,B)\subset S$. Let $E$ be any apartment containing $A$ and $B$. Then $\mbox{Con}(x,B)\subset \mbox{Con}(A,B)\subset E$. Since $\mbox{Con}(A,B)\subset Y$ and $\dim(A)=\dim(Y)$ we have that $\mbox{Con}(A,B)$ is contained in a Levi sphere $S'$ of $E$ with $\dim(S')=\dim(Y)$. Since $x=$op$_E y$ and $A\neq \mbox{proj}_xB$ ($A\not\in\mbox{Con}(x,B)$) we must have that $A$ is opposite $B$.
  \end{proof}
  
\subsection{Complete Reducibility in a twin building}

Let $X=(X_+,X_-)$ be the geometric realization of a twin building. We can give a definition of a completely reducible subcomplex which is analogous to the spherical case.

\begin{defs} A convex subcomplex $Y=(Y_+,Y_-)$ of a twin building $X=(X_+,X_-)$ is \emph{completely reducible} (or $Y$ is $X$-cr) if for every simplex  $y\in Y_\epsilon$ there is a simplex $y'\in Y_{-\epsilon}$ which is opposite $y$ for $\epsilon=+$ or $-$.\end{defs}

When $X$ is a twin building associated to a group $G$ with a twin $BN$-pair and $Y$ is the subcomplex stabilized by a subgroup $H$ of $G$, this definition of complete reducibility is equivalent to the one given by P.E. Caprace in \cite{Ca09} mentioned in the introduction.

We now list several propositions which Serre proves in \cite{Se04} for spherical buildings and whose proofs easily extend to the case of twin buildings where $Y$ is a convex subcomplex containing at least one spherical simplex in each component.  Note that by Lemma \ref{sphcha} this implies that every simplex of maximal dimension in $Y$ is spherical. We assume this is the case in what follows.  In the twin case, by a \emph{Levi sphere} $S$ we mean the convex hull of a pair of opposite spherical simplices $(s,s')$. Thus, $S=(S_+,S_-)$ is the support of $s$ in an apartment containing $s$ and $s'$. We continue to use the term ``sphere'' in order to be consistent with the spherical case and in the case of a Euclidean twin building, if we identify the twinned points at infinity in the two components the resulting space is homeomorphic to a sphere.

 This first proposition generalizes the equivalence of (a), (b) and (c) of Theorem \ref{Serthm}.

\begin{prop}Let $Y$ be a convex subcomplex of $X$. Then $Y$ is $X$-cr if and only if $Y$ contains a Levi sphere $S$ with $\dim(S)=\dim(Y)$.\end{prop}
  \begin{proof} If $Y$ is cr, $Y$ contains a pair of opposite simplices with the same dimension as $Y$. The convex closure of these two simplices is a Levi sphere. The proof of the converse is the same as that in Proposition \ref{Levisph}.\end{proof}
  
The next Lemma extends Lemma 2.6 in \cite{Se04} and follows from the gate property of twin buildings. Here $X_s$ is the link of $s$, and if $(s,s')$ is a pair of opposite simplices then the map proj$_{s'}$ from the set of simplices containing $s$ to the set of simplices containing $s'$ induces an isomorphism $X_s\rightarrow X_{s'}$ (\cite{Se04}). Let $S$ be the Levi sphere given by $s$ and $s'$. Then the \emph{building associated to $S$} is $X_s$ and will be written as $X_S$.

\begin{lem} Let $\{s,s'\}$ be a pair of opposite simplices, and let $t_1,t_2$ be two simplices of $X_s$. Let $t_1'$ be the simplex of $X_{s'}$ corresponding to proj$_{s'}t_1$. Then $t_1$ op $t_2$ in $X_s$ if and only if $t_1'$ op $t_2$ in $X$.\end{lem}
  \begin{proof}Since $X_s$ and $X_{s'}$ correspond to opposite residues of $X$, it suffices to show that given chambers $C_1$ and $C_2$ in $X_s$ then $C_1$ and $C_2$ are opposite in $X_s$ if and only if $C_1'$ op $C_2$ in $X$ where $C_1'$ is proj$_{s'}C_1$. This follows from the fact that $d^*(C_2,C_1')=d(C_2,C_1)-d^*(C_1,C_1')$ (Lemma 5.149 of [AB08]) and $d^*(C_1,C_1')=m$ where $m$ is the diameter of $X_s$ and $X_{s'}$.\end{proof}
  
The next proposition generalizes Proposition 2.5 in [Se04].
  
\begin{prop}\label{resbuild} Let $Y$ be a convex subcomplex of $X$, and let $S$ be a Levi sphere contained in $Y$. Let $X_S$ be the building associated to $S$, and let $Y_S$ be the subcomplex of $X_S$ defined by $Y$. Then $Y$ is $X$-cr if and only if $Y_S$ is $X_S$-cr.\end{prop}
  \begin{proof}This proof is the same as that for the spherical case.\end{proof}
  
The next proposition is Theorem 2.2 in [Se04]. The proof is similar to that in the spherical case.
  
\begin{prop}\label{oppver} $Y$ is $X$-cr if and only if for every spherical vertex $x$ of $Y$, there exists a vertex $x'$ of $Y$ with $x$ op $x'$.\end{prop}
  \begin{proof} Note that by the definition, if $Y$ is $X$-cr then every simplex has an opposite so in particular, every vertex has an opposite in $Y$. For the converse, let $y$ be spherical vertex of $Y$. Then by assumption there exists $y'\in Y$ which is opposite $y$. By Propositon \ref{resbuild}, it suffices to show that $Y_y$ is $X_y$-cr. We proceed by induction. Since dim$(X_y)=$dim$(X)-1$, it suffices to show that every vertex of $Y_y$ has an opposite in $Y_y$. Let $z$ be any vertex of $Y_y$. Then $z$ corresponds to a segment $yz$ with endpoints vertices $y$ and $z$ in $Y$. Let $z'$ be any vertex in $Y$ which is opposite $z$. Since $z$ op $z'$ we know that the convex hull of $yz$ and $z'$ has dimension one and is contained in $Y$. Consider proj$_yz'$. This is a one-dimensional simplex of $Y$ with one vertex being $y$. Let $z_1$ be the other vertex. Then by Lemma \ref{linkop} we have that $yz$ is opposite $yz'$ in $Y_y$. Hence $Y_y$ has the desired property and $Y$ is $X$-cr.  \end{proof}
  
\subsection{Complete Reducibility and the Building at Infinity}\label{crinfty}

In this section we assume that $X=(X_+,X_-)$ is a Euclidean twin building. Before proving the main result we need to state a couple of lemmas.
 
\begin{lem}\label{consecface} Let $x$ be a spherical vertex in $X_{\epsilon}$ and $A$ a simplex in $X_{-\epsilon}$ for $\epsilon=+$ or $-$. Then the convex hull of $x$ and $A$ in $X_{\epsilon}$ contains a conical cell.\end{lem}

\begin{proof} Let $E$ be a twin apartment containing the convex hull $\mbox{Con}(x,A)$ of $x$ and $A$. Let $B=$proj$_xA$. We know from the proof of Theorem \ref{convex} that $\mbox{Con}(x,A)_\epsilon$ is the intersection of roots $\alpha$ such that $A\not\in \alpha$ and $x\in \partial\alpha$, which are the roots $\alpha$ such that $B\in \alpha$ and $x\in \partial\alpha$. Hence if $x$ is a special vertex then $\mbox{Con}(x,A)_\epsilon$ is a conical cell. If $x$ is not a special vertex, then $\mbox{Con}(x,A)_\epsilon$ is  the union of a finite number of conical cells since $B$ is defined by a finite number of walls containing $x$. In either case, $\mbox{Con}(x,A)_\epsilon$ contains a sector-face.\end{proof}

\begin{lem}\label{linkapart}Let $E_0$ be a twin apartment of a twin building $X$. Let $M$ be a convex subcomplex of $E_0$ of dimension $m$. Let $y$ be a boundary simplex of $M$ and let $d\in M$ be the unique $m$-simplex having $y$ as a face. If $d'$ is any other $m$-simplex of $X$ having $y$ as a face, then there is a twin apartment containing $M$ and $d'$.\end{lem}

\begin{proof} Since $y$ is a boundary simplex of $M$ we can find a gallery $\Gamma:D_0,\ldots,D_n$ such that $d'$ is a face of $D_n$,  $D_i\cap M=y$ for $1\leq i\leq n$ and $D_0\cap M=d$. For each $1\leq i\leq n$ let $P_i=D_{i-1}\cap D_i$ and let $H_i$ be the wall in $E_{i-1}$ defined by $P_i$. Since $y\in H_i$ and $y$ is a boundary $(m-1)$-simplex of $M$ we have that $M$ is contained in a root, $\alpha_i$, of $E_{i-1}$ defined by $H_i$. By Lemma \ref{twnewapart}, there is a twin apartment $E_i$ containing $\alpha_i$ and $D_i$. Since $M\subset \alpha_i$ we have that $M\subset E_i$ for all $0\leq i\leq n$. Therefore, $E_n$ contains $d'$ and $M$.\end{proof}

\begin{thm}\label{infinitycr} Let $X=(X_+,X_-)$ be a Euclidean twin building and $Y=(Y_+,Y_-)$ a convex subcomplex of $X=(X_+,X_-)$. Let $I=(I_+,I_-)$ be the set of interior points in the buildings at infinity as in Section 4.2 and $Y^\infty=(Y^\infty_+,Y^\infty_-)$ the subcomplex of $I$ corresponding to $Y$. Then $Y$ is a completely reducible subcomplex of $X$ if and only if every simplex of maximal dimension in $Y^\infty$ has an interior opposite in $Y^\infty$.\end{thm}

\begin{proof} ($\Leftarrow$) Without loss of generality let $x$ be a vertex in $Y_+$ and $A$ a simplex of maximal dimension in $Y_-$. By Lemma \ref{consecface}, the convex hull of $x$ and $A$,  $\mbox{Con}(x,A)_+$ in $X_+$ contains a conical cell, and since $A$ is of maximal dimension in $Y_+$ so is $\mbox{Con}(x,A)$, call this conical cell $\mathfrak{A}$.  So $\mathfrak{A}$ corresponds to a simplex in $Y^\infty$, which has an interior opposite in $Y^\infty$ with a corresponding sectorface, $\mathfrak{A}'$. Hence $\mathfrak{A}'$ is parallel to a sectorface $\mathfrak{U}$ which is opposite to $\mathfrak{A}$ in some twin apartment $E$.

Let $y$ be the base point of $\mathfrak{U}$. By Proposition \ref{comapart} there is an apartment $E'$ containing $y$ and a conical cell $\mathfrak{A}''$ contained in $\mathfrak{A}'$ with the same direction as $\mathfrak{A}'$. Since $\mathfrak{U}$ is the unique conical cell parallel to $\mathfrak{A}''$ based at $y$ (Lemma 11.75 of \cite{AB08}), $\mathfrak{U}$ is in $E'$. Since $y$ is opposite $x$, dim$(\mbox{Con}(y,\mathfrak{A}''))=$dim$(\mbox{Con}(x, \mathfrak{A}''))$ and since $\mathfrak{A}''$ has maximal dimension in $Y$ dim$(\mbox{Con}(x, \mathfrak{A}''))=$dim$(\mathfrak{A}'')$. Hence, $y$ and $\mathfrak{U}$ are in the support of $\mathfrak{A}''$ in $E'$ and in particular, since $\A''$ and $\mathfrak{U}$ are parallel, their intersection $\mathfrak{U}'$ is a conical cell with the same direction (and therefore, the same dimension) as $\A''$. Since $\mathfrak{U}'\subseteq \mathfrak{U}=$op$_E \A$ we have op$_E \mathfrak{U}'\subseteq$op$_E\mathfrak{U}=\A$ (see Figure \ref{opsector}). Then $\mathfrak{U}'$ and op$_E\mathfrak{U}'$ are opposite conical cells in $Y$, so $Y$ contains the support of $\mathfrak{U}'$ in $E$ which is also the support of $\A$ in $E$. Therefore, there is a vertex in $Y$ which is opposite $x$ and $Y$ is completely reducible.  

\begin{figure}\begin{center}\includegraphics[scale=.5]{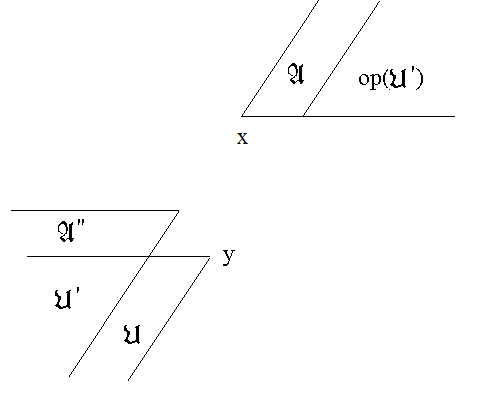}\caption{\label{opsector}Opposite conical cells.}\end{center}\end{figure}

($\Rightarrow$) Let $e$ be a simplex of maximal dimension in $Y^{\infty}$, and let $\A$ be a conical cell in $Y$ which corresponds to $e$.  It suffices to show that there is a twin apartment $E=(E_+,E_-)$ containing a conical cell, $\A'$, which is contained in $\A$ and has the same direction as $\A$ such that Supp$_E \A'\subset Y$.

Without loss of generality, let $\A$ be such a sector face in $Y_+$. Since $Y$ is completely reducible, $Y_-$ is not empty. Let $y$ be a simplex in $Y_-$ and let $m=$dim $(\A)$. By Lemma \ref{comapart} there is a twin apartment $E_0$ containing $y$ and a concical cell $\A'$ contained in $\A$ with the same direction as $\A$. If $y\in \mbox{Int}(\mbox{op}_{E_0}\A')$, then the convex hull $\mbox{Con}(y,\A')$ of $y$ and $\A'$ contains a simplex of dimension $m$ which is opposite to a simplex in $\A'$, hence $\mbox{Con}(y,\A')=\mbox{Supp}_{E_0}\A'$. 

So assume $y\not\in \mbox{Int}(\mbox{op}_{E_0}\A')$. Let $R^{(1)}:=\mbox{Con}(y,\A')$. Then $R_-^{(1)}\cap \mbox{Int}(\mbox{op}_{E_0}\A')$ is empty. Let $A$ be the base simplex of $\A'$. Let $y_1$ be a simplex of dimension $m-1$ such that dist$(y_1,\mbox{op}_{E_0}A)$ is minimal, hence $y_1$ is in the boundary of $R^{(1)}$. Let $d_1$ be the $m$-simplex of $R^{(1)}$ containing $y_1$. Since $Y$ is completely reducible there is a twin apartment $E_1$ containing $d_1$ such that Supp$_{E_1}d_1\subset Y$. In particular, there is an $m$-simplex $d_2\in Y$ such that $d_2\not\in R^{(1)}$ and $y_1\leq d_2$. By Lemma \ref{linkapart}, there is a twin apartment, $E_2$ containing $R^{(1)}$ and $d_2$.

Let $\Phi:E_1\rightarrow E_2$ be the isometry which fixes $E_1\cap E_2$. Since $\Phi$ preseves distance and codistance, then by choice of $d_1$, we know that the dist$(d_2,\mbox{op}_{E_2}A)<$ dist$(d_1,\mbox{op}_{E_1}A)$. Now let $R^{(2)}:=\mbox{Con}(d_2,R^{(1)})\subset Y$, and note that $R^{(2)}\subset E_2$. We continue this process noting that dist$(d_i,\mbox{op}_{E_i}A)<$ dist$(d_{i-1},\mbox{op}_{E_{i-1}}A)$, and $\A'\subset R^{(i)}\subset E_i$ for all $i$. So there is an $n>0$ such that $d_n=\mbox{op}_{E_n}A$ and $\A'\subset E_n$. Since $(d_n,\A')\subset Y$ and $\mbox{Con}(d_n,\A')=\mbox{Supp}_{E_n}\A'$ we have Supp$_{\Sigma_n}\A'\subset Y$.

\end{proof}

\begin{coro}\label{crcoro}Let $Y$ be a completely reducible subcomplex of a twin building $X$. Then $Y^\infty$ is a completely reducible subcomplex of the interior sub-building.\end{coro}

Note that the converse is not true because interior opposition is a stronger condition than opposition in the interior sub-building.

\begin{thm}\label{inftyvert} A convex subcomplex $Y$ is $X$-completely reducible if and only if every vertex in $Y^\infty$ has an interior opposite in $Y^\infty$.\end{thm}
\begin{proof} Note that if $Y$ is completely reducible then every simplex of maximal dimension in $Y^\infty$ has an interior opposite by the previous theorem, and hence every simplex of $Y^\infty$ has an interior opposite since each simplex is a face of simplex of maximal dimension. For the other direction, let $m:=\mbox{dim}(Y)$ let $x\in Y_+$ (the proof is identical if we instead let $x\in Y_-$) and let $A$ be a $m$-simplex of $Y_-$. Let $E_1$ be a twin apartment containing $x$ and $A$.

Let $R^{(1)}=\mbox{Con}(x,A)$ be the convex hull of $x$ and $A$ and let $y$ be a vertex of $R^{(1)}_-$ which has a minimal number of walls separating $y$ and op$_{E_1}x$. Let $H_1$ be a defining hyperplane of $R^{(1)}_-$ and let $\alpha_1$ be the corresponding root (note that $H_1$ separates $y$ and op$_{E_1}x$). Let $r$ be a ray on an edge of $R^{(1)}_+$ which is in the interior of $\alpha_1$. Then $r$ corresponds to a vertex, $e_+$, in ${Y^\infty}_+$ so there is a vertex, $e_-$, in ${Y^\infty}_-$ which is interior opposite $e_+$. Let $s$ be a ray in $Y_-$ corresponding to $e_-$. By Lemma \ref{comapart}, there is a twin apartment containing $y$ and a subray of $s$, and the convex hull of $y$ and this subray contains a ray parallel to $s$ and based at $y$. Let $d_1$ be the first 1-simplex of this ray. Since $s$ is parallel to a ray which is opposite $r$ we know that $d_1$ is not in $R^{(1)}_-$. Let $a_1$ be a $(m-1)$-simplex containing $y$ and in $\partial \alpha_1$.  Since $d_1\not\in R^{(1)}$, it is not in $\alpha_1$ so $b_1:=\mbox{proj}_{a_1}d_1$ is a $m$-simplex containing $a_1$ not in $R^{(1)}$. By Lemma \ref{linkapart} there is a twin apartment, $E_2$, containing $b_1$ and $R^{(1)}$. Let $R^{(2)}:=\mbox{Con}(b_1,R^{(1)})$. Since $b_1\not\in \alpha_1$ there is a vertex $y_2\in R^{(2)}$ such that the number of hyperplanes separating $y_2$ and op$_{E_2}x$ is strictly less than the number of hyperplanes separating $y$ and op$_{E_1}x$. Since this is a finite number we can repeat this process until there is a twin apartment $E_n$ such that there is a vertex $y_n$ of $R^{(n)}_-$ such that there are no hyperplanes separating $y_n$ and op$_{E_n}x$, hence $y_n$ is a vertex opposite $x$. Therefore, every vertex of $Y$ has an opposite vertex in $Y$ so $Y$ is completely reducible.
\end{proof} 

\subsubsection{Group theoretic consequence}

\begin{ex}\label{sln} Consider the group $G=SL_2(R)$ for $R=\mathbb{F}_2[t,t^{-1}]$. Let $K=\mathbb{F}_2(t)$ with $\nu_+$ the valuation on $K$ that gives the order at 0, and $\nu_-$ the valuation that gives the order at infinity. Let $A_{\pm}$ be the corresponding valuation rings. Then following \cite{AB08} section 6.12, we obtain a twin building $X=(X_+,X_-)$ where $X_{\pm}$ is isomorphic to a three regular tree with vertices corresponding to the $A_{\pm}$ lattice classes $[[t^{a_1}f_1,t^{a_2}f_2]]$ for any $K$-basis $\{f_1,f_2\}$ of $K^2$. The opposition involution takes an $A_+$ lattice class $[[t^{a_1}f_1,t^{a_2}f_2]]$ to the $A_-$ lattice class of the same symbol. Let $M=Re_1+Re_2$ in $K^2$ where $\{e_1,e_2\}$ is the standard basis of $K^2$. Then every $R$-basis of $M$ corresponds to a twin apartment. 

The building at infinity for $X_\epsilon$ is the set of ends of the tree. This is the spherical building associated to a vector space $V=\hat{K}^2$ with $\hat{K}$ being the completion of $K$ with respect to the valuation $\nu_\epsilon$ where each vertex corresponds to a subspace of $\hat{K}^2$. 

To find the interior vertices consider the interior ray $\mathfrak{r}$ given by the lattice classes $[[e_1,t^{n}e_2]]$ for $n>0$.  Following Section 11.8.6 of \cite{AB08}, the stabilizer in $SL_2(K)$ of the corresponding end in $X_+^\infty$ is the upper triangular subgroup. Hence the stabilizer in $G=SL_2(R)$ of this end is the upper triangular subgroup of $SL_2(R)$ which stabilizes the $R$-submodule of $M$ given by $Re_1$. In $X_-^\infty$ the stabilizer in $SL_2(R)$ of the end corresponding to the ray $\mathfrak{r}'$ given by the lattice classes $[[e_1,t^{n}e_2]]$ is the lower triangular subgroup. This subgroup stabilizes the $R$-submodule of $M$ given by $Re_2$. 

Similarly, every interior vertex of $X_\pm$ corresponds to an $R$-submodule of $M$ given by $Rf_1$ such that $M=Rf_1\oplus Rf_2$ where $\{f_1,f_2\}$ is an $R$-basis of $M$. In other words, the interior vertices correspond to the rank 1 $R$-submodules of $M$ which are $R$ direct summands of $M$. Since the rank one $R$-submodules correspond to the one dimensional subspaces of $K^2$, the interior sub-building is the same as the sub-building corresponding to $K$.

Two interior vertices $e$ and $e'$ corresponding to $R$-submodules $M_1$ and $M_2$ are interior opposite if and only if $M=M_1\oplus M_2$ as $R$-modules. Note that $M_1=Rf_1$ and $M_2=Rf_2$ since $M_1$ and $M_2$ are rank one so $M=Rf_1\oplus Rf_2$ and $\{f_1,f_2\}$ is an $R$-basis of $M$. Then the ray $[[f_1,t^nf_2]]$ for $n>0$ in $X_+$ gives rise to the end corresponding to $Rf_1$ which is $e$. The ray with the same description in $X_-$ gives rise to the end corresponding to $Rf_2$ which is $e'$. Hence these two rays are opposite so $e$ and $e'$ are interior opposite. 

Conversely, if $e$ and $e'$ are interior opposite there is a twin apartment given by an $R$-basis $\{f_1,f_2\}$ such that $e$ and $e'$ arise from opposite rays, with description $[[t^af_1,t^nf_2]]$ where $a$ is a fixed integer and $n$ increases from 0. Then $M_1=Rt^af_1=Rf_1$ and $M_2=Rf_2$. Since $\{f_1,f_2\}$ is an $R$-basis, $M_1\oplus M_2=Rf_1\oplus Rf_2=M$. Note that this condition is stronger than that for being opposite in the building associated to $K$.




\end{ex}
The above discussion can be generalized to $G=SL_n[R]$ for $R=k[t,t^{-1}]$ for any field $k$. The construction of the corresponding twin building is given in Section 6.12 of \cite{AB08}, which generalizes the method above. If $\{f_1,\ldots,f_n\}$ is an $R$ basis for the free $R$-module $R^n$ then there is a twin apartment, $E$, whose vertices are the lattice classes $[[t^{a_1}f_1,\ldots, t^{a_n}f_n]]$ for $a_1,\ldots,a_n\in \Z$. 

The buildings at infinity $X_\pm$ are the spherical buildings associated to the vector space $\hat{K}^n$ where $K=k(t)$. The vertices correspond to proper subspaces and the simplices correspond to chains of subspaces.

Let $E=(E_+,E_-)$ be a twin apartment associated to the $R$-basis of $R^n$ $\{f_1,\ldots,f_2\}$. The rays in $E$ are the sequences of lattices classes described below:

Let $A$ be a subset of $\{1,\ldots,n\}$ and let $L_m:=[[t^{a_{1,m}}f_1,\ldots, t^{a_{n,m}}f_n]]$ where $a_{i,m}=a_{i,0}$ if $i\in A$ and $a_{i,m}=a_{i,0}+m$ if $i\not\in A$ for $m\in \N$. Then the sequence $\{L_m\}$ of vertices and the one dimensional simplices connecting them is a ray in $E$. The lattice class $[[t^{a_1}f_1,\ldots,t^{-a_k}f_k,\ldots, t^{a_n}f_n]]$ for $a_k>0$ is equivalent to the lattice class $[[t^{a_1+a_k}f_1,\ldots,f_k,\ldots, t^{a_n+a_k}f_n]]$ so the ray given by the sequence of lattice classes $L'_m=[[t^{a_{1,m}}f_1,\ldots, t^{a_{n,m}}f_n]]$ where $a_{i,m}=a_{i,0}$ if $i\not\in A$ and $a_{i,m}=a_{i,0}+m$ if $i\in A$ for $m\in \N$ is the reversal of the ray given by $\{L_m\}$.

Let $\mathfrak{r}$ be the ray in $E_+$ given by the sequence $\{L_m\}$. The associated interior vertex of $I_+$ corresponds to the free $R$-submodule of $R^n$ with basis $\{f_i\}_{i\in A}$. Let $\mathfrak{r}'$ be the ray in $E_-$ given by the sequence $\{L_m\}$, so that $\mathfrak{r}'=$op$_E\mathfrak{r}$. The associated interior vertex of $I_-$ corresponds to the free $R$-submodule of $R^n$ with basis $\{f_i\}_{i\not\in A}$. So we can say that two vertices $e,e'$ of $I=(I_+,I_-)$  corresponding to $R$-submodules $M_1$ and $M_2$ are opposite if and only if $R^n=M_1\oplus M_2$ as $R$-modules using the same argument as in the rank 2 case above.

Let $\Gamma\leq G=SL_n(R)$ be a subgroup. Corollary \ref{crcoro} implies that if $\Gamma$ is a completely reducible subgroup of $G$ then $\Gamma$ is a completely reducible subgroup of $SL_n(K)$. The following proposition follows from the above discussion and Theorem \ref{inftyvert}.

\begin{prop}The subgroup $\Gamma$ is completely reducible if and only if every $\Gamma$-invariant $R$-submodule of $R^n$ which is an $R$ direct summand of $R^n$ has a $\Gamma$-invariant $R$-complement.\end{prop}

\begin{prop} Let $K=k(t)$ and let $\Gamma$ be a completely reducible subgroup of $G$. Then $R^n=M_1\oplus\cdots\oplus M_k$ where each $M_i$ is a $\Gamma$-invariant $R$ submodule such that $K\otimes_R M_i$ is irreducible in $K^n$.\end{prop}

\begin{proof} If $K^n$ has no proper nonzero $\Gamma$-invariant submodules then $K^n$ is irreducible and $R^n$ is a $\Gamma$-invariant $R$ submodule such that $K\otimes_RR^n$ is irreducible. So let $S_1$ be a proper nonzero $\Gamma$-invariant irreducible submodule of $K^n$. Then $M_1:=S_1\cap R^n$ is a $\Gamma$-invariant $R$-submodule of $R^n$ which is a $R$-direct summand of $R^n$. Since $\Gamma$ is completely reducible, there exists a $M_1'$ which is a $\Gamma$-invariant submodule such that $R^n=M_1\oplus M_1'$. If $S_1':=K\otimes_R M_1'$ is irreducible we are done. If not, let $S_2$ be a proper nonzero irreducible $\Gamma$-invariant submodule of $S_1'$ and let $M_2=S_2\cap R^n$. Then $M_1\oplus M_2$ is a $\Gamma$-invariant submodule of $R^n$ which is a $R$-direct summand so it has a $\Gamma$-invariant $R$-complement $M_2'$. We can continue this process which terminates since $n<\infty$ to get $R^n=M_1\oplus M_2\oplus\cdots\oplus M_k$ with $M_i$ $\Gamma$-invariant and $S_i=K\otimes_R M_i$ irreducible by construction.\end{proof}

\subsection{Questions}
\begin{quest} Thick Euclidean buildings of rank greater than 3 (and rank equal to 3 if the building at infinity is Moufang) have been classified and correspond to certain groups.  What are the group theoretic consequences of Theorems \ref{infinitycr} and \ref{inftyvert} to these other groups?\end{quest}

\begin{quest} Recently, Caprace and L{\'e}cureux wrote a paper \cite{CL09} on compactifications of arbitrary buildings which extend the notion of the building at infinity for Euclidean buildings to more general buildings. Since Kac-Moody groups are in general not Euclidean, it would be interesting to see if there is a condition for complete reducibility on a twin building consisting of pairs with this compactification and extend the results of Theorems \ref{infinitycr} and \ref{inftyvert} to the non-Euclidean case.\end{quest}

\bibliography{cites}
\bibliographystyle{alpha}

\end{document}